\documentclass[sn-mathphys]{sn-jnl}

\jyear{2022}%

\theoremstyle{thmstyleone}%
\newtheorem{theorem}{Theorem}
\newtheorem{proposition}[theorem]{Proposition}%
\newtheorem{corollary}[theorem]{Corollary}

\theoremstyle{thmstyletwo}%
\newtheorem{remark}{Remark}%

\theoremstyle{thmstylethree}%

\def\ds{\displaystyle}
\def\Rset{\mathbb{R}}

\begin{document}

\title[The role of permanently resident populations]{The role of permanently resident populations in the two-patches SIR model with commuters}


\author*[1]{\fnm{Alain} \sur{Rapaport}}\email{alain.rapaport@inrae.fr}

\author[1]{\fnm{Ismail} \sur{Mimouni}}\email{ismail.mimouni@yahoo.com}

\affil*[1]{\orgdiv{MISTEA}, \orgname{Universit\'e Montpellier, INRAE, Institut Agro}, \orgaddress{\street{place Viala}, \city{Montpellier}, \postcode{34060}, \country{France}}}

\abstract{We consider a two-patches SIR model where communication occurs thru commuters, distinguishing explicitly permanently resident populations from commuters populations. We give an explicit formula of the reproduction number, and show how the proportions of permanently resident populations impact it. We exhibit non-intuitive situations for which allowing commuting from a safe territory to another one where the transmission rate is higher can reduce the overall epidemic threshold and avoid an outbreak.}




\maketitle

\section{Introduction}
\label{secintro}

Since the pioneer work of Kermack and McKendrick \cite{Kermack27}, the SIR model has been very popular in epidemiology, as the basic model for infectious diseases with direct transmission (e.g.~\cite{Anderson91}).  It retakes great importance nowadays due to the recent coronavirus pandemic. While early models were not spatialized, the importance of accounting for spatial heterogeneity has been often reported in the literature (see, e.g.~\cite{Angulo79, Sattenspiel95, Keeling04, Keeling07, Kelly16, Li21}). However, different mechanisms come into play to explain the spatial spreading of a disease. Although diffusion appears to be a natural process to describe the local propagation of an infectious agent among a population, which leads to models with partial differential equations \cite{Murray03}, it appears to be not well suited for describing long distance spreading. In particular, transportation between cities comes into the picture as a major source of rapid spreading among non-homogeneous populations \cite{Arino03, Arino07, Takeuchi07, Liu13, Mpolya14, Chen14, Yin20, ToctoErazo21,Lipshtat21}. Meta-populations or multi-patches models are then more appropriate to describe the spatial characteristics of the propagation \cite{Wang03, Wang04,  Arino06, Gao07, Arino09}, as already well considered in ecology \cite{Hanski99, McArthur01}. These models require a precise description of the movements between patches, which are most of the time assumed to be linear and thus encoded into a connection matrix \cite{Arino06, Arino09}. Typically one obtains a system of ordinary differential equations on a graph, which couples the communication dynamics with the epidemiological one. 

For diseases spreading among human populations living in different cities, {\em commuters} (individuals housing in a city, traveling regularly for short periods in a neighboring city, and coming back to their home city) play a crucial role in the disease propagation among territories \cite{Keeling02, Keeling04, Keeling07, Mpolya14, Yin20}. Such coupling between patches have been already considered in the literature, distinguishing among populations $N_i$ attached to a city $i$ the sub-population $N_{ii}$ present in its permanent housing from other sub-populations $N_{ij}$ temporary present in another city $j \neq i$ (it can be also seen as multi-groups models as in \cite{Clancy96, Guo06, Iggidr12}). However, such models explicitly assume that the whole population housing in a given city can potentially commute to another one. We believe that this is not always fully realistic and that a sub-population that never (or very rarely) moves to another city should be distinguished from the sub-population that visits at a regular basis another city.
Therefore, we consider an extension of such models, which explicitly takes into consideration two kinds of movement: an Eulerian one which describes the flow between patches that mixes populations, and an Lagrangian one which assigns home locations of individuals, as described in the more general framework \cite{Citron_etal2021}.
The study of this extension, which has not yet been analyzed analytically in the literature, to our knowledge, and how it impacts the disease spreading, is the primary objective of the present work. For this purpose, we establish an analytical expression of the reproduction number (as the epidemic threshold formerly introduced and analyzed in \cite{Diekmann90,vandenDriessche02,Diekmann07,Dhirasakdanon07}) for the two patches case (that is also valid for the particular case when the whole populations travel, for which the exact expression of the reproduction number has not been yet provided in the literature).

We also had in mind to consider heterogeneity among territories when disease transmission differs from one city to another one. Typically, non-pharmaceutical interventions (such as reducing physical distance in the population) could be applied with different strength in each city, providing distinct transmission rates.
When one territory being isolated presents a higher reproduction number than the other territory, it can be considered as a {\em core group} in the epidemiological terminology \cite{Hadeler-Castillo-Chavez1995,Brunham1997}, and commuters contribute then to spread the epidemics in both territories. We aim at analyzing more precisely how the proportions of commuters in each city can increase or decrease the overall reproduction number. Intuitively, one may believe that the best way to reduce the spreading is to encourage commuters from the city with the lowest transmission rate not  to travel to the other city, and on the opposite to encourage as much as possible commuters from the other city to spend time in the safer city. Indeed, we shall see that this is not always true... The second objective of the present work is thus to study the minimization of the epidemic threshold of the two-patches model with respect to these proportions, depending on the commuting rates. This analysis can potentially serve for decisions making to prevent epidemic outbreak (as in \cite{Knipl16} for instance).

The paper is organized as follows. In the next section, we present the complete model in dimension $18$ and give some preliminaries. Section \ref{secR0} is devoted to the analysis of the asymptotic behavior of the solutions of the model. We give and demonstrate an explicit expression of the reproduction number, introducing four relevant quantities $q_{ij}$ ($i, j = 1,2$). In a corollary, we also give an alternative way of computation, which is useful in the following. In Section \ref{secmin}, we study the minimization of the reproduction number with respect to the proportions of commuters in each patch. Finally, Section \ref{secillu} gives a numerical illustration of the results, considering two territories with intrinsic basic reproduction numbers lower and higher than one. We depict the relative sizes of the permanently resident populations that can avoid the outbreak of the epidemic depending on the commuting rates, and discuss the various cases. We end with a conclusion.

\section{The model}
\label{secmodel}

We follow the modeling of commuters proposed in \cite{Keeling02} between two patches (such as cities or territories), but here we consider in addition that a part of the population in each patch do not commute (the {\em permanently resident} sub-population).
We consider populations of size $N_i$ whose home belongs to a patch $i \in \{1,2\}$, structured in three groups:
\begin{enumerate}[i.]
	\item permanently resident, being all the time in patch $i$, whose population size is denoted $N_{ir}$,
	\item commuters to patch $j$, but located in patch $i$ at time $t$, of population size denoted $N_{ii}$,
	\item commuters to patch $j$ and located in patch $j$ at time $t$, of population size denoted $N_{ij}$.
\end{enumerate}
We shall denote $N_{ic}=N_{ii}+N_{ij}$ the size of the total population of commuters with their home in patch $i$. The individuals commutes to patch $j$ at a rate $\lambda_i$ with a return rate $\mu_i$. For each group $g \in \{ir,ii,ij\}$ we denote by $S_g$, $I_g$, $R_g$ the sizes of susceptible, infected and recovered sub-populations.
This modeling implicitly assumes that at any time there is no individual out the territories, that is traveling time is negligible. This assumption is therefore only valid for adjoining territories with short transportation times (by train, road...). It would not be valid between distant territories connected for example by boat or plane with non-negligible crossing times. In this case, it would be necessary to consider additional nodes of {\em in-transit} populations, as it has been considered for example in \cite{Colizza_etal2013,Patil2021} or in \cite{Ruan2015} where distance between nodes are explicitly taken into consideration. This would of course complicates the model and its study.

\medskip

We consider the SIR model assuming that the recovery parameter $\gamma$ is identical everywhere while the transmission rate $\beta_i$ depends on the patch $i$ but is identical among each group. 
Typically lifestyle and hygienic measures may differ between two cities, implying different values of $\beta$. Moreover, if two cities are on both sides of the border between two countries, the strength of non-pharmaceutical interventions are likely to be different, as is was for instance the case between European countries during the SARS-2 outbreak.
The model is written as follows (with $i \neq j$ in $\{1,2\}$).

\begin{flalign*}
\dot S_{ir}&=-\beta_i S_{ir}\frac{I_{ir}+I_{ii}+I_{ji}}{N_{ir}+N_{ii}+N_{ji}},\\
\dot I_{ir}&=\beta_i S_{ir}\frac{I_{ir}+I_{ii}+I_{ji}}{N_{ir}+N_{ii}+N_{ji}}-\gamma I_{ir},\\
\dot R_{ir}&=\gamma I_{ir},\\
\dot S_{ii}&=-\beta_i S_{ii}\frac{I_{ir}+I_{ii}+I_{ji}}{N_{ir}+N_{ii}+N_{ji}}-\lambda_i S_{ii}+\mu_i S_{ij},\\
\dot I_{ii}&=\beta_i S_{ii}\frac{I_{ir}+I_{ii}+I_{ji}}{N_{ir}+N_{ii}+N_{ji}}-\gamma I_{ii}-\lambda_i I_{ii}+\mu_i I_{ij},\\
\dot R_{ii}&=\gamma I_{ii}-\lambda_i R_{ii}+\mu_i R_{ij},\\
\dot S_{ij}&=-\beta_j S_{ij}\frac{I_{jr}+I_{jj}+I_{ij}}{N_{jr}+N_{jj}+N_{ij}}+\lambda_i S_{ii}-\mu_i S_{ij},\\
\dot I_{ij}&=\beta_j S_{ij}\frac{I_{jr}+I_{jj}+I_{ij}}{N_{jr}+N_{jj}+N_{ij}}-\gamma I_{ij}+\lambda_i I_{ii}-\mu_i I_{ij},\\
\dot R_{ij}&=\gamma I_{ij}+\lambda_i R_{ii}-\mu_i R_{ij}
\end{flalign*}
Parameters $\lambda_i$, $\mu_i$ represent switching rates of populations $i$, leaving home and returning. This modeling implicitly assumes that movements between territories are not synchronized, as often considered in multi-city models (see e.g.~ \cite{Sattenspiel95,Keeling02,Arino03,Wang03,Wang04,Keeling04,Arino06,Takeuchi07,Keeling07,Liu13,Chen14}). Note that we also consider, in all generality, that commuting is asymmetrical (i.e.~$\lambda_1$ and  $\lambda_2$ may be different, as well as $\mu_1$, $\mu_2$). Typically, each territory may offer different activities that attract commuters from the other territory, and thus different mean sojourn times.
One can check that the population sizes $N_{ir}$, 
and $N_{ic}$ 
are constant.
Moreover $N_{ii}$, 
$N_{ij}$ 
fulfill the system of equations

\begin{equation*}
\left\{\begin{array}{l}
\dot N_{ii} = -\lambda_i N_{ii} + \mu_i N_{ij},\\
\dot N_{ij} = \lambda_i N_{ii} - \mu_i N_{ij}
\end{array}\right.
\end{equation*}
whose solutions verify

\begin{equation}
\label{Nlimit}
\lim_{t \to +\infty} N_{ii}(t)=\bar N_{ii}:=\frac{\mu_i}{\lambda_i+\mu_i}N_{ic}, \quad \lim_{t \to +\infty} N_{ij}(t)=\bar N_{ij}:=\frac{\lambda_i}{\lambda_i+\mu_i}N_{ic}
\end{equation}
We shall assume that populations are already balanced at initial time i.e.~that one has $N_{ii}=\bar N_{ii}$, $N_{ij}=\bar N_{ij}$ (constant). For simplicity, we shall drop the notation $\bar{ }\,$ in the following,
and denote

\begin{equation*}
N_{ip}:=N_{ir}+N_{ii}+N_{ji}
\end{equation*}
which represents the (constant) size of the total population present in patch $i$.

\section{The epidemic threshold}
\label{secR0}

We denote the vectors 

\begin{equation*}
I=(I_{1r},I_{11},I_{12},I_{2r},I_{22},I_{21})^\top, \quad S=(S_{1r},S_{11},S_{12},S_{2r},S_{22},S_{21})^\top
\end{equation*}
and consider the state vector 

\begin{equation*}
X=\left[\begin{array}{c}I\\S\end{array}\right]
\end{equation*}
which belongs to the invariant domain

\begin{equation*}
{\cal D} : = \{ X \in \Rset_+^{12} ; \; \mathbb{M}X\leq \mathbb{N}\}
\end{equation*}
where $\mathbb{N}$ is the vector

\begin{equation*}
\mathbb{N} =(N_{1r},N_{11},N_{12},N_{2r},N_{22},N_{21})^\top
\end{equation*}
and $\mathbb{M}$ the  $6\times 12$ matrix which consists in the concatenation of the identity matrix $\mathbb{I}_6$ of dimension $6\times 6$
\[
\mathbb{M}=[\mathbb{I}_6 , \mathbb{I}_6] 
\]
The disease free equilibrium is defined as

\begin{equation*}
X^\star=\left[\begin{array}{c}0\\\mathbb{N}\end{array}\right]
\end{equation*}
Let ${\cal R}_i$ be the intrinsic reproduction number in the patch $i$ (i.e.~when there is no connection between patches), that is

\begin{equation*}
{\cal R}_i:=\frac{\beta_i}{\gamma} .
\end{equation*}
We give now an explicit expression of the epidemic threshold when the two patches communicates via commuters. 

\begin{proposition}
	\label{mainprop}
	Let
	
	\begin{equation}
	\label{R012}
	{\cal R}_{1,2}:=\frac{ q_{11} + q_{22} + \sqrt{ (q_{22} - q_{11})^2 + 4q_{12}q_{21}}}{2}
	\end{equation}
	where
	
	\begin{equation}
	\label{qij}
	\left\{\begin{array}{l}
	q_{11}= {\cal R}_1 \left(\frac{N_{1r}}{N_{1p}}+ \frac{N_{11}}{N_{1p}}\frac{\gamma+\mu_1}{\gamma+\lambda_1+\mu_1}+
	\frac{N_{21}}{N_{1p}}\frac{\gamma+\lambda_2}{\gamma+\lambda_2+\mu_2}\right)\\
	\noalign{\medskip} q_{22}= {\cal R}_2 \left(\frac{N_{2r}}{N_{2p}}+ \frac{N_{22}}{N_{2p}}\frac{\gamma+\mu_2}{\gamma+\lambda_2+\mu_2}+
	\frac{N_{12}}{N_{2p}}\frac{\gamma+\lambda_1}{\gamma+\lambda_1+\mu_1}\right)\\
	\noalign{\medskip} q_{21}= {\cal R}_1 \left( \frac{N_{11}}{N_{1p}}\frac{\lambda_1}{\gamma+\lambda_1+\mu_1}+
	\frac{N_{21}}{N_{1p}}\frac{\mu_2}{\gamma+\lambda_2+\mu_2}\right)\\
	\noalign{\medskip} q_{12}= {\cal R}_2 \left( \frac{N_{12}}{N_{2p}}\frac{\mu_1}{\gamma+\lambda_1+\mu_1}+
	\frac{N_{22}}{N_{2p}}\frac{\lambda_2}{\gamma+\lambda_2+\mu_2}\right)
	\end{array}\right.
	\end{equation}
	Then, one has the following properties.
	\begin{enumerate}[i.]
		\item If ${\cal R}_{1,2}>1$, then $X^\star$ is unstable.
		
		\item If ${\cal R}_{1,2}<1$, then $X^\star$ is exponentially stable with respect to the variable\footnote{We refer to \cite{Vorotnikov98} for the definition of partial stability.}  $I$.
		
		\item If ${\cal R}_1={\cal R}_2:={\cal R}$, then ${\cal R}_{1,2}={\cal R}$.
	\end{enumerate}
\end{proposition}

\medskip

\begin{proof}
	Write the dynamics of $X$ as $\dot X=f(X)$. The Jacobian matrix $J$ of $f$ at $X^\star$ is of the form
	\[
	J=\left[\begin{array}{cc} A & 0\\
	\star & B \end{array}\right] \mbox{ with } A=F-V
	\]
	where
	\[
	F=  \left[ \begin {array}{cccccc} \beta_1\frac{N_{1r}}{N_{1p}}&\beta_1\frac{N_{1r}}{N_{1p}}&0&0&0&\beta_1\frac{N_{1r}}{N_{1p}}\\ 
	\noalign{\medskip}\beta_1\frac{N_{11}}{N_{1p}}&\beta_1\frac{N_{11}}{N_{1p}}&0&0&0&\beta_1\frac{N_{11}}{N_{1p}}\\
	\noalign{\medskip}0&0&\beta_2\frac{N_{12}}{N_{2p}}&\beta_2\frac{N_{12}}{N_{2p}}&\beta_2\frac{N_{12}}{N_{2p}}&0\\ \noalign{\medskip}0&0&\beta_2\frac{N_{2r}}{N_{2p}}&\beta_2\frac{N_{2r}}{N_{2p}}&\beta_2\frac{N_{2r}}{N_{2p}}&0\\ \noalign{\medskip}0&0&\beta_2\frac{N_{22}}{N_{2p}}&\beta_2\frac{N_{22}}{N_{2p}}&\beta_2\frac{N_{22}}{N_{2p}}&0\\ \noalign{\medskip}\beta_1\frac{N_{21}}{N_{1p}}&\beta_1\frac{N_{21}}{N_{1p}}&0&0&0&\beta_1\frac{N_{21}}{N_{1p}}\end {array} \right],
	\]
	\smallskip
	\[
	V=\left[ \begin {array}{cccccc} \gamma&0&0&0&0&0\\ \noalign{\medskip}0&
	\gamma+\lambda_1&-\mu_1&0&0&0\\ \noalign{\medskip}0&-\lambda_1&\gamma+\mu_1&0&0
	&0\\ \noalign{\medskip}0&0&0&\gamma&0&0\\ \noalign{\medskip}0&0&0&0&\gamma+\lambda_2&-\mu_2\\ \noalign{\medskip}0&0&0&0&-\lambda_2&\gamma+\mu_2
	\end {array} \right] 
	\]
	and
	\[
	B=\left[ \begin {array}{cccccc} 0&0&0&0&0&0\\ 
	\noalign{\medskip}0&-\lambda_1&\mu_1&0&0&0\\ 
	\noalign{\medskip}0&\lambda_1&-\mu_1&0&0&0\\ 
	\noalign{\medskip}0&0&0&0&0&0\\ 
	\noalign{\medskip}0&0&0&0&-\lambda_2&\mu_2\\ 
	\noalign{\medskip}0&0&0&0&\lambda_2&-\mu_2
	\end {array} \right] 
	\]
	Note that $F$ is a non-negative matrix and $V$ is a non-singular M-matrix. We recall (see for instance from \cite{vandenDriessche02}) that one has the property
	\[
	\max Re(Spec(A))\underset{\displaystyle >}{<}0 \Longleftrightarrow \rho(FV^{-1})\underset{\displaystyle >}{<}1
	\]
	The computation of the matrix $M:=FV^{-1}$ gives the following expression
	{\tiny
	\[
	M=\left[ \begin {array}{cccccc} 
	
	{\cal R}_1\frac{N_{1r}}{N_{1p}}&
	{\cal R}_1\frac{N_{1r}(\gamma+\mu_1)}{N_{1p}(\gamma+\lambda_1+\mu_1)}&
	{\cal R}_1\frac{N_{1r}\mu_1}{N_{1p}(\gamma+\lambda_1+\mu_1)}&
	0&
	{\cal R}_1\frac{N_{1r}\lambda_2}{N_{1p}(\gamma+\lambda_2+\mu_2)}&
	{\cal R}_1\frac{N_{1r}(\gamma+\lambda_2)}{N_{1p}(\gamma+\lambda_2+\mu_2)}\\ 
	
	\noalign{\medskip}
	{\cal R}_1\frac{N_{11}}{N_{1p}}&
	{\cal R}_1\frac{N_{11}(\gamma+\mu_1)}{N_{1p}(\gamma+\lambda_1+\mu_1)}&
	{\cal R}_1\frac{N_{11}\mu_1}{N_{1p}(\gamma+\lambda_1+\mu_1)}&
	0&
	{\cal R}_1\frac{N_{11}\lambda_2}{N_{1p}(\gamma+\lambda_2+\mu_2)}&
	{\cal R}_1\frac{N_{11}(\gamma+\lambda_2)}{N_{1p}(\gamma+\lambda_2+\mu_2)}\\ 
	
	\noalign{\medskip}
	0&
	{\cal R}_2\frac{N_{12}\lambda_1}{N_{2p}(\gamma+\lambda_1+\mu_1)}&
	{\cal R}_2\frac{N_{12}(\gamma+\lambda_1)}{N_{2p}(\gamma+\lambda_1+\mu_1)}&
	{\cal R}_2\frac{N_{12}}{N_{2p}}&
	{\cal R}_2\frac{N_{12}(\gamma+\mu_2)}{N_{2p}(\gamma+\lambda_2+\mu_2)}&
	{\cal R}_2\frac{N_{12}\mu_2}{N_{1p}(\gamma+\lambda_2+\mu_2)}\\ 
	
	\noalign{\medskip}
	0&
	{\cal R}_2\frac{N_{2r}\lambda_1}{N_{2p}(\gamma+\lambda_1+\mu_1)}&
	{\cal R}_2\frac{N_{2r}(\gamma+\lambda_1)}{N_{2p}(\gamma+\lambda_1+\mu_1)}&
	{\cal R}_2\frac{N_{2r}}{N_{2p}}&
	{\cal R}_2\frac{N_{2r}(\gamma+\mu_2)}{N_{2p}(\gamma+\lambda_2+\mu_2)}&
	{\cal R}_2\frac{N_{2r}\mu_2}{N_{1p}(\gamma+\lambda_2+\mu_2)}\\ 
	
	\noalign{\medskip}
	0&
	{\cal R}_2\frac{N_{22}\lambda_1}{N_{2p}(\gamma+\lambda_1+\mu_1)}&
	{\cal R}_2\frac{N_{22}(\gamma+\lambda_1)}{N_{2p}(\gamma+\lambda_1+\mu_1)}&
	{\cal R}_2\frac{N_{22}}{N_{2p}}&
	{\cal R}_2\frac{N_{22}(\gamma+\mu_2)}{N_{2p}(\gamma+\lambda_2+\mu_2)}&
	{\cal R}_2\frac{N_{22}\mu_2}{N_{1p}(\gamma+\lambda_2+\mu_2)}\\ 
	
	\noalign{\medskip}
	{\cal R}_1\frac{N_{21}}{N_{1p}}&
	{\cal R}_1\frac{N_{21}(\gamma+\mu_1)}{N_{1p}(\gamma+\lambda_1+\mu_1)}&
	{\cal R}_1\frac{N_{21}\mu_1}{N_{1p}(\gamma+\lambda_1+\mu_1)}&
	0&
	{\cal R}_1\frac{N_{21}\lambda_2}{N_{1p}(\gamma+\lambda_2+\mu_2)}&
	{\cal R}_1\frac{N_{21}(\gamma+\lambda_2)}{N_{1p}(\gamma+\lambda_2+\mu_2)}
	
	\end {array} \right] 
	\]}
	Let us consider the diagonal matrix
	\[
	D:=\left[  \begin{array}{cccccc}
	{\cal R}_1\frac{N_{1r}}{N_{1p}} & \\
	& {\cal R}_1\frac{N_{11}}{N_{1p}} \\
	& & {\cal R}_2\frac{N_{12}}{N_{2p}}\\
	& & & {\cal R}_2\frac{N_{2r}}{N_{2p}}\\
	& & & & {\cal R}_2\frac{N_{22}}{N_{2p}}\\
	& & & & & {\cal R}_1\frac{N_{21}}{N_{1p}}
	\end{array}\right]
	\]
	and the matrix $Q=D^{-1} M D$, whose computation gives the expression
{\tiny
	\[
	Q=\left[ \begin {array}{cccccc} 
	
	{\cal R}_1\frac{N_{1r}}{N_{1p}}&
	{\cal R}_1\frac{N_{11}(\gamma+\mu_1)}{N_{1p}(\gamma+\lambda_1+\mu_1)}&
	{\cal R}_2\frac{N_{12}\mu_1}{N_{2p}(\gamma+\lambda_1+\mu_1)}&
	0&
	{\cal R}_2\frac{N_{22}\lambda_2}{N_{2p}(\gamma+\lambda_2+\mu_2)}&
	{\cal R}_1\frac{N_{21}(\gamma+\lambda_2)}{N_{1p}(\gamma+\lambda_2+\mu_2)}\\ 
	
	\noalign{\medskip}
	{\cal R}_1\frac{N_{1r}}{N_{1p}}&
	{\cal R}_1\frac{N_{11}(\gamma+\mu_1)}{N_{1p}(\gamma+\lambda_1+\mu_1)}&
	{\cal R}_2\frac{N_{12}\mu_1}{N_{2p}(\gamma+\lambda_1+\mu_1)}&
	0&
	{\cal R}_2\frac{N_{22}\lambda_2}{N_{2p}(\gamma+\lambda_2+\mu_2)}&
	{\cal R}_1\frac{N_{21}(\gamma+\lambda_2)}{N_{1p}(\gamma+\lambda_2+\mu_2)}\\ 
	
	\noalign{\medskip}
	0&
	{\cal R}_1\frac{N_{11}\lambda_1}{N_{1p}(\gamma+\lambda_1+\mu_1)}&
	{\cal R}_2\frac{N_{12}(\gamma+\lambda_1)}{N_{2p}(\gamma+\lambda_1+\mu_1)}&
	{\cal R}_2\frac{N_{2r}}{N_{2p}}&
	{\cal R}_2\frac{N_{22}(\gamma+\mu_2)}{N_{2p}(\gamma+\lambda_2+\mu_2)}&
	{\cal R}_1\frac{N_{21}\mu_2}{N_{1p}(\gamma+\lambda_2+\mu_2)}\\ 
	
	\noalign{\medskip}
	0&
	{\cal R}_1\frac{N_{11}\lambda_1}{N_{1p}(\gamma+\lambda_1+\mu_1)}&
	{\cal R}_2\frac{N_{12}(\gamma+\lambda_1)}{N_{2p}(\gamma+\lambda_1+\mu_1)}&
	{\cal R}_2\frac{N_{2r}}{N_{2p}}&
	{\cal R}_2\frac{N_{22}(\gamma+\mu_2)}{N_{2p}(\gamma+\lambda_2+\mu_2)}&
	{\cal R}_1\frac{N_{21}\mu_2}{N_{1p}(\gamma+\lambda_2+\mu_2)}\\

	\noalign{\medskip}
	0&
	{\cal R}_1\frac{N_{11}\lambda_1}{N_{1p}(\gamma+\lambda_1+\mu_1)}&
	{\cal R}_2\frac{N_{12}(\gamma+\lambda_1)}{N_{2p}(\gamma+\lambda_1+\mu_1)}&
	{\cal R}_2\frac{N_{2r}}{N_{2p}}&
	{\cal R}_2\frac{N_{22}(\gamma+\mu_2)}{N_{2p}(\gamma+\lambda_2+\mu_2)}&
	{\cal R}_1\frac{N_{21}\mu_2}{N_{1p}(\gamma+\lambda_2+\mu_2)}\\ 
	
	\noalign{\medskip}
	{\cal R}_1\frac{N_{1r}}{N_{1p}}&
	{\cal R}_1\frac{N_{11}(\gamma+\mu_1)}{N_{1p}(\gamma+\lambda_1+\mu_1)}&
	{\cal R}_2\frac{N_{12}\mu_1}{N_{2p}(\gamma+\lambda_1+\mu_1)}&
	0&
	{\cal R}_2\frac{N_{22}\lambda_2}{N_{2p}(\gamma+\lambda_2+\mu_2)}&
	{\cal R}_1\frac{N_{21}(\gamma+\lambda_2)}{N_{1p}(\gamma+\lambda_2+\mu_2)} 
	
	\end {array} \right] 
	\]}
	The matrix $Q$ is non-negative and irreducible. By Perron-Frobenius Theorem (see for instance \cite{BermanPlemmons}), this matrix admits a unique positive eigenvector (up to a scalar multiplication) that corresponds to the simple (positive) eigenvalue $\ell=\rho(Q)=\rho(M)$.
	
	Note that the rank of $Q$ is two. We posit
	\[
	Y=(1,1,0,0,0,1)^\top, \quad Z=(0,0,1,1,1,0)^\top
	\]
	and define $Q_Y$, $Q_Z$ the first and third lines, respectively, of the matrix $Q$.
	Then, for any vector $X \in \Rset^6$, on has one has $QX=(Q_Y X)Y+(Q_Z X)Z$. We look for an positive eigenvector $X$ of the form $X=\alpha Y + (1-\alpha) Z$ with $\alpha \in (0,1)$. One has then
	
	\begin{align}
	\label{cond1}
	\nonumber
	QX & = \alpha QY + (1-\alpha) QZ\\
	\nonumber
	&  = \alpha\big( (Q_Y Y)Y+(Q_Z Y)Z  \big) + (1-\alpha) \big( (Q_Y Z)Y+(Q_Z Z)Z  \big)\\
	& =\big(  \alpha (Q_Y Y) + (1-\alpha)  (Q_Y Z) \big) Y + \big(  \alpha (Q_Z Y) + (1-\alpha)  (Q_Z Z) \big) Z
	\end{align}
	On the other hand, as $X$ is an eigenvector, one has
	\begin{equation}
	\label{cond2}
	QX=\ell X = \alpha \ell Y + (1-\alpha)\ell Z
	\end{equation}
	The vectors $Y$ and $Z$ being orthogonal, one obtains from \eqref{cond1}-\eqref{cond2} the conditions
	\begin{equation}
	\label{sys_alpha}
	\left\{\begin{array}{l}
	\alpha Q_Y Y + (1-\alpha)  Q_Y Z = \alpha \ell\\
	\alpha Q_Z Y + (1-\alpha)  Q_Z Z  = (1-\alpha)\ell
	\end{array}\right.
	\end{equation}
	Let $r=\frac{1-\alpha}{\alpha}$. Eliminating $\ell$ in the two previous equations, $r$ is the positive solution of the polynomial
	\[
	r^2 Q_Y Z + r(Q_Y Y -Q_Z Z) -Q_Z Y=0
	\]
	and $\ell=Q_Y Y + r Q_Y Z$. One obtains the expression of the eigenvalue
	\[
	\ell= \frac{ Q_Y Y + Q_Z Z + \sqrt{ (Q_Y Y - Q_Z Z)^2 + 4 (Q_Y Z)(Q_Z Y)}}{2}
	\]
	Finally, from the expression of $Q$, one gets
	\[
	\begin{array}{l}
	q_{11}=Q_Y Y = {\cal R}_1 \left(\frac{N_{1r}}{N_{1p}}+ \frac{N_{11}}{N_{1p}}\frac{\gamma+\mu_1}{\gamma+\lambda_1+\mu_1}+
	\frac{N_{21}}{N_{1p}}\frac{\gamma+\lambda_2}{\gamma+\lambda_2+\mu_2}\right)\\
	\noalign{\medskip}
	q_{22}=Q_Z Z = {\cal R}_2 \left(\frac{N_{2r}}{N_{2p}}+ \frac{N_{22}}{N_{2p}}\frac{\gamma+\mu_2}{\gamma+\lambda_2+\mu_2}+
	\frac{N_{12}}{N_{2p}}\frac{\gamma+\lambda_1}{\gamma+\lambda_1+\mu_1}\right)\\
	\noalign{\medskip}
	q_{21}=Q_Z Y = {\cal R}_1 \left( \frac{N_{11}}{N_{1p}}\frac{\lambda_1}{\gamma+\lambda_1+\mu_1}+
	\frac{N_{21}}{N_{1p}}\frac{\mu_2}{\gamma+\lambda_2+\mu_2}\right)\\
	\noalign{\medskip}
	q_{12}=Q_Y Z = {\cal R}_2 \left( \frac{N_{12}}{N_{2p}}\frac{\mu_1}{\gamma+\lambda_1+\mu_1}+
	\frac{N_{22}}{N_{2p}}\frac{\lambda_2}{\gamma+\lambda_2+\mu_2}\right)\\
	\end{array}
	\]
	and thus $\ell={\cal R}_{1,2}$, which is exactly $\rho(M)$.
	
	\medskip
	
	i. When ${\cal R}_{1,2}>1$, the matrix $A$ has at least one eigenvalue with positive real part and the matrix $J$ as well. The equilibrium $X^\star$ is thus unstable on ${\cal D}$.
	
	\medskip
	
	ii. When ${\cal R}_{1,2}<1$, the matrix $A$ is Hurwitz, but $X^\star$ is not an hyperbolic equilibrium. However, on can write the dynamics of the vector $I$ as an non-autonomous system
	\[
	\dot I = g(t,I):=\left(\begin{array}{c}
	\beta_1 S_{1r}(t)\frac{I_{1r}+I_{11}+I_{21}}{N_{1p}}-\gamma I_{1r}\\
	\noalign{\medskip} \beta_1 S_{11}(t)\frac{I_{1r}+I_{11}+I_{21}}{N_{1p}}-(\gamma +\lambda_1)I_{11}+\mu_1I_{12}\\
	\noalign{\medskip} \beta_2 S_{12}(t)\frac{I_{2r}+I_{22}+I_{12}}{N_{2p}}+\lambda_1I_{11}-(\gamma +\mu_1)I_{12}\\
	\noalign{\medskip}  \beta_2 S_{2r}(t)\frac{I_{2r}+I_{22}+I_{12}}{N_{2p}}-\gamma I_{2r}\\
	\noalign{\medskip} \beta_2 S_{22}(t)\frac{I_{2r}+I_{22}+I_{12}}{N_{2p}}-(\gamma +\lambda_2)I_{22}+\mu_2I_{21}\\
	\noalign{\medskip} \beta_1 S_{21}(t)\frac{I_{1r}+I_{11}+I_{21}}{N_{1p}}+\lambda_2I_{22}-(\gamma +\mu_2)I_{21}
	\end{array}\right)
	\]
	Note that this dynamics is cooperative and as for any $t\geq 0$ one has $S_{ij}(t)\leq N_{ij}$ for $ij \in \{1r,11,12,2r,22,21\}$, one gets
	\[
	g(t,I) \leq \bar g(I):=AI, \; I \geq 0
	\]
	Therefore, any solution $I(\cdot)$ of $\dot I=g(t,I)$ with $I(0)=I_0\geq 0$ verifies $0\leq I(t) \leq \bar I(t)$ for any $t\geq 0$, where $\bar I(\cdot)$ is solution of the linear dynamics $\dot{\bar I}=\bar g(\bar I)$ with $\bar I(0)=I_0$. As $A$ is Hurwitz, we conclude that $X^\star$ is exponentially stable with respect to $I$, which proves point ii.
	
	\medskip
	
	iii. For the particular case  ${\cal R}_1={\cal R}_2:={\cal R}$, the transpose of the matrix $M$ writes
	{\tiny 
		\[
	M^\top={\cal R}
	\left[ \begin {array}{cccccc} 
	
	\frac{N_{1r}}{N_{1p}}& 
	\frac{N_{11}}{N_{1p}}&
	0&
	0&
	0&
	\frac{N_{21}}{N_{1p}}\\
	
	\noalign{\medskip}
	\frac{N_{1r}(\gamma+\mu_1)}{N_{1p}(\gamma+\lambda_1+\mu_1)}&
	\frac{N_{11}(\gamma+\mu_1)}{N_{1p}(\gamma+\lambda_1+\mu_1)}&
	\frac{N_{12}\lambda_1}{N_{2p}(\gamma+\lambda_1+\mu_1)}&
	\frac{N_{2r}\lambda_1}{N_{2p}(\gamma+\lambda_1+\mu_1)}&
	\frac{N_{22}\lambda_1}{N_{2p}(\gamma+\lambda_1+\mu_1)}&
	\frac{N_{21}(\gamma+\mu_1)}{N_{1p}(\gamma+\lambda_1+\mu_1)}\\
	
	\noalign{\medskip}
	\frac{N_{1r}\mu_1}{N_{1p}(\gamma+\lambda_1+\mu_1)}&
	\frac{N_{11}\mu_1}{N_{1p}(\gamma+\lambda_1+\mu_1)}&
	\frac{N_{12}(\gamma+\lambda_1)}{N_{2p}(\gamma+\lambda_1+\mu_1)}&
	\frac{N_{2r}(\gamma+\lambda_1)}{N_{2p}(\gamma+\lambda_1+\mu_1)}&
	\frac{N_{22}(\gamma+\lambda_1)}{N_{2p}(\gamma+\lambda_1+\mu_1)}&
	\frac{N_{21}\mu_1}{N_{1p}(\gamma+\lambda_1+\mu_1)}\\
	
	\noalign{\medskip}
	0&
	0&
	\frac{N_{12}}{N_{2p}}&
	\frac{N_{2r}}{N_{2p}}&
	\frac{N_{22}}{N_{2p}}&
	0\\
	
	\noalign{\medskip}
	\frac{N_{1r}\lambda_2}{N_{1p}(\gamma+\lambda_2+\mu_2)}&
	\frac{N_{11}\lambda_2}{N_{1p}(\gamma+\lambda_2+\mu_2)}&
	\frac{N_{12}(\gamma+\mu_2)}{N_{2p}(\gamma+\lambda_2+\mu_2)}&
	\frac{N_{2r}(\gamma+\mu_2)}{N_{2p}(\gamma+\lambda_2+\mu_2)}&
	\frac{N_{22}(\gamma+\mu_2)}{N_{2p}(\gamma+\lambda_2+\mu_2)}&
	\frac{N_{21}\lambda_2}{N_{1p}(\gamma+\lambda_2+\mu_2)}\\
	
	\noalign{\medskip}
	\frac{N_{1r}(\gamma+\lambda_2)}{N_{1p}(\gamma+\lambda_2+\mu_2)}& 
	\frac{N_{11}(\gamma+\lambda_2)}{N_{1p}(\gamma+\lambda_2+\mu_2)}&
	\frac{N_{12}\mu_2}{N_{1p}(\gamma+\lambda_2+\mu_2)}&
	\frac{N_{2r}\mu_2}{N_{1p}(\gamma+\lambda_2+\mu_2)}&
	\frac{N_{22}\mu_2}{N_{1p}(\gamma+\lambda_2+\mu_2)}&
	\frac{N_{21}(\gamma+\lambda_2)}{N_{1p}(\gamma+\lambda_2+\mu_2)}
	
	\end {array} \right] 
	\]}
	One can check that one has $M^\top U={\cal R}U$ where $U=(1,1,1,1,1,1)^\top$. As $U$ is a positive vector, we deduce from the Perron-Frobenius Theorem that one has $\rho(M)=\rho(M^T)={\cal R}$, which ends the proof.
\end{proof}

	\begin{remark}
		\label{remark}
		More generally, the next-generation matrix $M=FV^{-1}$ can be shown to have a rank equal to the number $n$ of patches, and that its Perron vector can be expressed as a convex combination of a family of orthogonal vectors in the image of $M$. This implies that the positive eigenvalue of $M$ (i.e. the reproduction number) is also the positive eigenvalue of the $n$-dimensional positive matrix given by the decomposition of the image of this vectors by the matrix $M$. \\
		Alternatively, one may consider the epidemic spread in a virgin population as a Markov process, to determine the expected numbers of secondary cases in each patch, and obtain this $n\times n$ matrix, as described in \cite{Diekman_etal2013}. This method consists in a {\em first-step analysis} by determining the mean residence times of an infected individual of each group in each of the patches. Then, for a given patch the expected numbers of new infected present in each path are given by the products of the mean residence times by the transmission rate, averaged by the constant distribution given in \eqref{Nlimit}. \\
		This explains why the formula \eqref{R012} takes the expression of a root of the characteristic polynomial of a $2$ by $2$ matrix.
	\end{remark}
\begin{remark}
	The explicit expression \eqref{R012} of the epidemic threshold given in Proposition \ref{mainprop} is also relevant in absence of permanently resident populations, which has not been yet provided explicitly in the literature (up to our knowledge).
\end{remark}

\begin{corollary}
	\label{corominmax}
	One has
	\[
	\min\left({\cal R}_1,	{\cal R}_2\right) \leq {\cal R}_{1,2} \leq
	\max\left({\cal R}_1,	{\cal R}_2\right) .
	\]
\end{corollary}

\begin{proof}
	Denote by $M({\cal R}_1,{\cal R}_2)$ the matrix $FV^{-1}$ for the parameters ${\cal R}_1$, ${\cal R}_2$, and let ${\cal R}_-:=\min\left({\cal R}_1,{\cal R}_2\right)$, ${\cal R}_+:=\max\left({\cal R}_1,{\cal R}_2\right)$.
	From the expression of the non-negative matrices $M$, one gets
	\[
	M({\cal R}_-,{\cal R}_-) \leq M({\cal R}_1,{\cal R}_2) \leq M({\cal R}_+,{\cal R}_+)
	\]
	which implies (see for instance \cite{BermanPlemmons}) the inequalities
	\[
	\rho(M({\cal R}_-,{\cal R}_-)) \leq \rho(M({\cal R}_1,{\cal R}_2)) \leq \rho(M({\cal R}_+,{\cal R}_+))
	\]
	and thus
	\[
	{\cal R}_- \leq {\cal R}_{1,2} \leq {\cal R}_+ .
	\]
\end{proof}

Alternatively, the number ${\cal R}_{1,2}$ can be determined as follows.

\begin{corollary}
	\label{maincoro}
	Assume ${\cal R}_2>{\cal R}_1$. Then, one has
	
	\begin{equation}
	\label{R012alpha}
	{\cal R}_{1,2}=\alpha {\cal R}_1 + (1-\alpha){\cal R}_2
	\end{equation}
	where $\alpha \in [0,1)$ is the smallest root of the polynomial
	
	\begin{equation*}
	P(\alpha)=\alpha^2({\cal R}_2-{\cal R}_1)-\alpha({\cal R}_2-{\cal R}_1+q_{12}+q_{21})+q_{12}
	\end{equation*}
\end{corollary}

\begin{proof}
	One can check, from expressions \eqref{qij}, that one has $q_{11}+q_{21}={\cal R}_1$ and $q_{22}+q_{12}={\cal R}_2$. Then, from \eqref{sys_alpha}, one get 
	
	\begin{equation}
	\label{R012alpha}
	{\cal R}_{1,2}=l=\alpha {\cal R}_1 + (1-\alpha){\cal R}_2
	\end{equation}
	where $\alpha$ is a root of the polynomial $P$ obtained from \eqref{sys_alpha} by eliminating $l$, that is
	
	\begin{equation*}
	P(\alpha)=\alpha^2({\cal R}_2-{\cal R}_1)-\alpha({\cal R}_2-{\cal R}_1+q_{12}+q_{21})+q_{12}
	\end{equation*}
	From Corollary \ref{corominmax}, we know that $\alpha$ belongs to $[0,1]$.
	Note that one has $P(0)=q_{12}\geq 0$ and $P(1)=-q_{21}\leq 0$. Therefore, when ${\cal R}_2-{\cal R}_1>0$,
	$P$ admits exactly one root in $[0,1)$ and another one in $[1,\to)$. However, if $\alpha=1$ one should have $q_{21}=0$ and thus $\lambda_1=0$, $\mu_2=0$, which implies $N_{11}=N_{1c}$, $N_{12}=0$, $N_{22}=0$, $N_{21}=N_{2c}$. Then, one obtains $q_{11}={\cal R}_1$, $q_{22}={\cal R}_2$ and from the expression \eqref{R012} on gets ${\cal R}_{1,2}=\max({\cal R}_1,{\cal R}_2)={\cal R}_2$ which contradicts $\alpha=1$. We conclude that $\alpha$ belongs to $[0,1)$ and is thus the smallest root of $P$.
\end{proof}

\begin{remark}
	When there is no communication between patches (that is $N_{1r}=N_{1p}=N_1$, $N_{2r}=N_{2p}=N_2$), one has $q_{21}=0$ and $q_{12}=0$. If ${\cal R}_2>{\cal R}_1$, resp.~${\cal R}_1>{\cal R}_2$, one has $\alpha=0$, resp.~$\alpha=1$, which gives
	\[
	{\cal R}_{1,2} =\max({\cal R}_1,{\cal R}_2) .
	\]
\end{remark}

We look now for a characterization of the minimum value of the threshold ${\cal R}_{1,2}$.

\section{Minimization of the epidemic threshold}
\label{secmin}

In this section, we assume that the mixing is fast compared to the recovery rate (as its is often considered in the literature), which amounts to have numbers $\lambda_i$, $\mu_i$ large compared to $\gamma$. Our objective is to study how the proportions of commuters in the populations impact the value of ${\cal R}_{1,2}$.

Given ${\cal R}_1$, ${\cal R}_2$, we consider the approximation $\tilde {\cal R}_{1,2}$ of the threshold ${\cal R}_{1,2}$ which consists in keeping $\gamma=0$ in the  expressions \eqref{qij}.
For convenience, we posit the numbers
\[
\eta_i:= \frac{\lambda_i}{\lambda_i+\mu_i} \in (0,1) \qquad (i=1,2)
\]
One has a first result about the variations of $\tilde {\cal R}_{1,2}$ with respect to $N_{1c}$, $N_{2c}$.

\begin{proposition}
	\label{prop_dR12}
	Fix parameters $N_i$, $\beta_i$, $\gamma$, $\lambda_i$, $\mu_i$ ($i=1,2$) such that
	${\cal R}_2>{\cal R}_1$.
	\begin{enumerate}[i.]
		
		\item For any $N_{1c} \in (0,N_1)$, the map $N_{2c} \mapsto \tilde {\cal R}_{1,2}(N_{1c},N_{2c})$ is decreasing.
		
		\item The map $N_{1c} \mapsto \tilde {\cal R}_{1,2}(N_{1c},N_{2c})$ is increasing at $(N_{1c},N_{2c})$ when
		\begin{equation}
		\label{cond+a}
		\eta_2(1-\eta_2)N_{2c} > (1-\eta_1)(N_2-\eta_2N_{2c})
		\end{equation}
		
		\item 
		The map $N_{1c} \mapsto \tilde {\cal R}_{1,2}(N_{1c},N_{2c})$ is increasing, resp.~decreasing, at $(N_{1c},N_{2c})$ if the numbers $A$ and $B$ are negative, resp.~positive, where
		
		\begin{align*}
		& A:={\cal R}_2 \frac{\frac{N_2}{2}-\eta_1(\frac{1}{2}-\eta_1)N_{1c}-(\frac{3}{2}-\eta_2)\eta_2N_{2c}}{N_2-\eta_2N_{2c}+\eta_1N_{1c}}\\
		& \hspace{30mm}
		-{\cal R}_1 \frac{\frac{N_1}{2}-(\frac{3}{2}-\eta_1)\eta_1N_{1c}-\eta_2(\frac{1}{2}-\eta_2)N_{2c}}{N_1-\eta_1N_{1c}+\eta_2N_{2c}}	, \\
		& B := {\cal R}_2
		\frac{(1-\eta_1)(N_2-\eta_2N_{2c})-\eta_2(1-\eta_2)N_{2c}}{(N_2-\eta_2N_{2c}+\eta_1N_{1c})^2}\\
		& \hspace{30mm} 
		-{\cal R}_1\frac{(1-\eta_1)(N_1+\eta_2N_{2c})+\eta_2(1-\eta_2)N_{2c}}{(N_1-\eta_1N_{1c}+\eta_2N_{2c})^2}
		\end{align*}
	\end{enumerate}	
\end{proposition}

\begin{proof}
	Following Corollary \ref{maincoro}, one has
	
	\begin{equation}
	\label{R12approx}
	\tilde {\cal R}_{1,2}=\tilde \alpha {\cal R}_1 + (1-\tilde \alpha){\cal R}_2
	\end{equation}
	where $\tilde\alpha$ is the smallest root of the polynomial
	
	\begin{equation*}
	\tilde P(\alpha)=\alpha^2({\cal R}_2-{\cal R}_1)-\alpha({\cal R}_2-{\cal R}_1+\tilde q_{12}+\tilde q_{21})+\tilde q_{12}
	\end{equation*}
	where $\tilde q_{12}$, $\tilde q_{21}$ are the approximations of $q_{12}$, $q_{21}$ defined in \eqref{qij}.
	Let us note that one can write
	$N_{ii}=(1-\eta_i)N_{ic}$, $N_{ij}=\eta_i N_{ic}$ (for $j \neq i$) and also
	$N_{ip}=N_i-\eta_i N_{ic}+\eta_j N_{jc}$, which leads to the following expressions of $\tilde q_{12}$, $\tilde q_{21}$ 
	
	\begin{equation}
	\label{qijapprox}
	\tilde q_{21}={\cal R}_1 \frac{(1-\eta_1)\eta_1N_{1c}+\eta_2(1-\eta_2)N_{2c}}{N_1-\eta_1N_{1c}+\eta_2N_{2c}}, \quad 
	\tilde q_{12}={\cal R}_2 \frac{\eta_1(1-\eta_1)N_{1c}+(1-\eta_2)\eta_2N_{2c}}{N_2-\eta_2N_{2c}+\eta_1N_{1c}}
	\end{equation}
	For simplicity, we shall drop the notation $\tilde{ }\;$ in the rest of the proof.
	Note than $\alpha$ being the smallest root of $P$, it verifies
	
	\begin{equation}
	\label{upperboundalpha}
	\alpha < \frac{{\cal R}_2-{\cal R}_1+q_{12}+q_{21}}{2({\cal R}_2-{\cal R}_1)}
	\end{equation}
	Let us differentiate the equality $P(\alpha)=0$ with respect to $q_{12}$ and $q_{21}$:
	
	\begin{align*}
	& 2\alpha \frac{\partial\alpha}{\partial q_{12}}({\cal R}_2-{\cal R}_1) -\frac{\partial\alpha}{\partial q_{12}}({\cal R}_2-{\cal R}_1+q_{12}+q_{21}) -\alpha +1=0\\
	& 2\alpha \frac{\partial\alpha}{\partial q_{21}}({\cal R}_2-{\cal R}_1) -\frac{\partial\alpha}{\partial q_{21}}({\cal R}_2-{\cal R}_1+q_{12}+q_{21}) -\alpha =0
	\end{align*}
	which gives
	
	\begin{align*}
	&  \frac{\partial\alpha}{\partial q_{12}} =\frac{1-\alpha}{{\cal R}_2-{\cal R}_1+q_{12}+q_{21}-2\alpha({\cal R}_2-{\cal R}_1)}\\
	& \frac{\partial\alpha}{\partial q_{21}} =\frac{-\alpha}{{\cal R}_2-{\cal R}_1+q_{12}+q_{21}-2\alpha({\cal R}_2-{\cal R}_1)}
	\end{align*}
	Then, one can write
	
	\begin{equation*}
	\frac{\partial \alpha}{\partial N_{ic}}  = \frac{\partial \alpha}{\partial q_{12}}\frac{\partial q_{12}}{\partial N_{ic}}+
	\frac{\partial \alpha}{\partial q_{21}}\frac{\partial q_{21}}{\partial N_{ic}}\\
	= \frac{(1-\alpha)\frac{\partial q_{12}}{\partial N_{ic}}-\alpha\frac{\partial q_{21}}{\partial N_{ic}}}{{\cal R}_2-{\cal R}_1+q_{12}+q_{21}-2\alpha({\cal R}_2-{\cal R}_1)} \qquad (i=1, 2)
	\end{equation*}
	and from inequality \eqref{upperboundalpha}, we obtain that the signs of the derivatives $\frac{\partial \alpha}{\partial N_{ic}}$ are given by the sign of the numbers
	
	\begin{equation}
	\label{sigma}
	\sigma_i := (1-\alpha)\frac{\partial q_{12}}{\partial N_{ic}}-\alpha\frac{\partial q_{21}}{\partial N_{ic}} \qquad (i=1,2)
	\end{equation}
	We begin by the dependency with respect to $N_{2c}$. One has first
	\[
	\frac{\partial q_{12}}{\partial N_{2c}}={\cal R}_2\eta_2
	\frac{(1-\eta_2)( N_2+\eta_1N_{1c})+\eta_1(1-\eta_1)N_{1c}}{(N_2+\eta_1N_{1c}-\eta_2N_{2c})^2} >0
	\]
	Note that one has
	
	\begin{equation}
	\label{q21fq12}
	q_{21}=\frac{{\cal R}_1(N_2+\eta_1N_{1c}-\eta_2N_{2c})}{{\cal R}_2(N_1-\eta_1N_{1c}+\eta_2N_{2c})}q_{12}
	\end{equation}
	and thus
	\[
	\frac{\partial q_{21}}{\partial N_{2c}}=\frac{{\cal R}_1(N_2+\eta_1N_{1c}-\eta_2N_{2c})}{{\cal R}_2(N_1-\eta_1N_{1c}+\eta_2N_{2c})}\frac{\partial q_{12}}{\partial N_{2c}} -
	\frac{{\cal R}_1 \eta_2 (N_1+N_2)}{{\cal R}_2(N_1-\eta_1N_{1c}+\eta_2N_{2c})^2}q_{12}
	\]
	Then, one gets the inequality
	\[
	\sigma_2 > \left( 1 - \alpha - \alpha \frac{{\cal R}_1(N_2+\eta_1N_{1c}-\eta_2N_{2c})}{{\cal R}_2(N_1-\eta_1N_{1c}+\eta_2N_{2c})} \right)\frac{\partial q_{12}}{\partial N_{2c}}
	\]
	On another hand, one gets from $P(\alpha)=0$ the inequality
	\[
	(1-\alpha)q_{12}-\alpha q_{21}= \alpha(1-\alpha)({\cal R}_2-{\cal R}_1)> 0
	\]
	and with \eqref{q21fq12}
	\[
	(1-\alpha)q_{12}-\alpha q_{21}= \left( 1 - \alpha - \alpha \frac{{\cal R}_1(N_2+\eta_1N_{1c}-\eta_2N_{2c})}{{\cal R}_2(N_1-\eta_1N_{1c}+\eta_2N_{2c})} \right)q_{12} >0
	\]
	We then conclude that $\sigma_2$ is positive, and from \eqref{R12approx} we deduce that the map $N_{2c} \mapsto  {\cal R}_{1,2}$ is decreasing. This proves the point i.
	
	\medskip
	
	We study now the dependency with respect to $N_{1c}$. A calculation of the partial derivative gives
	
	\begin{equation}
	\label{dq12dN1c}
	\frac{\partial q_{12}}{\partial N_{1c}}={\cal R}_2\eta_1
	\frac{(1-\eta_1)(N_2-\eta_2N_{2c})-\eta_2(1-\eta_2)N_{2c}}{(N_2-\eta_2N_{2c}+\eta_1N_{1c})^2}
	\end{equation}
	and
	
	\begin{equation}
	\label{dq21dN1c}
	\frac{\partial q_{21}}{\partial N_{1c}}={\cal R}_1\eta_1
	\frac{(1-\eta_1)(N_1+\eta_2N_{2c})+\eta_2(1-\eta_2)N_{2c}}{(N_1-\eta_1N_{1c}+\eta_2N_{2c})^2} >0
	\end{equation}
	When $\frac{\partial q_{12}}{\partial N_{1c}}<0$, we can conclude that $\sigma_1$ is negative and $ {\cal R}_{1,2}$ is thus increasing with respect to $N_{1c}$. This condition is equivalent to \eqref{cond+a}. This proves the point ii.
	When this last condition is not satisfied, having $\frac{\partial q_{12}}{\partial N_{1c}}<\frac{\partial q_{21}}{\partial N_{1c}}$  with  $\alpha>\frac{1}{2}$ is another sufficient condition to obtain $\sigma_1<0$ from expression \eqref{sigma}. However, having $\alpha  >\frac{1}{2}$ amounts to have $P(\frac{1}{2})>0$, that is
	\[
	\frac{{\cal R}_2-{\cal R}_1}{4}-\frac{{\cal R}_2-{\cal R}_1+q_{12}+q_{21}}{2}+q_{12}>0
	\]
	or equivalently
	\[
	\frac{{\cal R}_2}{2}-q_{12} < \frac{{\cal R}_1}{2} - q_{21}
	\]
	One can check that this last condition is equivalent to $A<0$ and that the condition $\frac{\partial q_{12}}{\partial N_{1c}}<\frac{\partial q_{21}}{\partial N_{1c}}$ is equivalent to $B<0$.
	In the same manner, having $A>0$ and $B>0$ implies $\alpha<\frac{1}{2}$ and $\frac{\partial q_{12}}{\partial N_{1c}}>\frac{\partial q_{21}}{\partial N_{1c}}$, which is a sufficient condition to have $\sigma_1>0$, and thus $ {\cal R}_{1,2}$ increasing with respect to $N_{1c}$. This proves the point iii.
\end{proof}

This result suggests that the map $N_{1c} \mapsto \tilde {\cal R}_{1,2}(N_{1c},N_{2c})$ is not necessarily monotonic, differently to the map $N_{2c} \mapsto \tilde {\cal R}_{1,2}(N_{1c},N_{2c})$. We show now that the possibilities of its variations are limited. 

\begin{proposition}
	\label{propmonotony}
	Under hypotheses of Proposition \ref{prop_dR12}, for each $N_{2c} \in (0,N_2)$ the map $N_{1c} \mapsto \tilde {\cal R}_{1,2}(N_{1c},N_{2c})$ possesses one of the three properties
	\begin{enumerate}[a.]
		\item it is decreasing on $(0,N_1)$,
		\item it is increasing on $(0,N_1)$,
		\item there exists $N_{1c}^\star \in (0,N_1)$ such that it is decreasing on $(0,N_{1c}^\star)$ and increasing on $(N_{1c}^\star,N_1)$.
	\end{enumerate}
\end{proposition}

\begin{proof}
	Fix $N_{2c} \in (0,N_2)$. If the map $N_{1c} \mapsto \tilde {\cal R}_{1,2}(N_{1c},N_{2c})$ is not monotonic, there exists  $\hat N_{1c} \in (0,N_1)$ such that $\frac{\partial \tilde {\cal R}_{1,2}}{\partial N_{1c}}(\hat N_{1c},N_{2c})=0$. 	For simplicity, we shall drop the notation $\tilde{ }\;$ in the rest of the proof.
	Following the proof of Proposition \ref{prop_dR12}, one has
	${\cal R}_{1,2}=\alpha {\cal R}_1 + (1-\alpha){\cal R}_2$ with
	\[
	\frac{\partial \alpha}{\partial N_{1c}}=\frac{(1-\alpha)\frac{\partial q_{12}}{\partial N_{1c}}-\alpha\frac{\partial q_{21}}{\partial N_{1c}}}{{\cal R}_2-{\cal R}_1+q_{12}+q_{21}-2\alpha({\cal R}_2-{\cal R}_1)}:=\frac{\sigma_1}{\nu}
	\]
	where $\nu>0$. Therefore, one has $\frac{\partial \alpha}{\partial N_{1c}}=0$ and $\sigma_1=0$ at $N_{1c}=\hat N_{1c}$, and thus
	\[
	\left.\frac{\partial^2 \alpha}{\partial N_{1c}^2}\right\vert_{N_{1c}=\hat N_{1c}}  =
	\left.\frac{\frac{\partial \sigma_1}{\partial N_{1c}}}{\nu}\right\vert_{N_{1c}=\hat N_{1c}} =
	\left.\frac{(1-\alpha)\frac{\partial^2 q_{12}}{\partial N_{1c}^2}-\alpha\frac{\partial^2 q_{21}}{\partial N_{1c}^2}}{\nu}\right\vert_{N_{1c}=\hat N_{1c}}
	\]
	From expressions \eqref{dq12dN1c} and \eqref{dq21dN1c}, a calculation of the partial derivatives gives
	\[
	\frac{\partial^2q_{12}}{\partial N_{1c}^2}=\frac{-2\eta_1\frac{\partial q_{12}}{\partial N_{1c}}}{N_1-\eta_1N_{1c}+\eta_2N_{2c}} , \quad
	\frac{\partial^2q_{21}}{\partial N_{1c}^2}=\frac{2\eta_1 \frac{\partial q_{21}}{\partial N_{1c}}}{N_2-\eta_2N_{2c}+\eta_1N_{1c}}
	\]
	where $\frac{\partial q_{21}}{\partial N_{1c}}>0$ and from $\sigma_1=0$ one gets $\frac{\partial q_{12}}{\partial N_{1c}}>0$ for $N_{1c}=\hat N_{1c}$. Finally, one obtains
	\[
	\frac{\partial^2 {\cal R}_{1,2}}{\partial N_{1c}^2}(\hat N_{1c},N_{2c})=-({\cal R}_2-{\cal R}_1)\frac{\partial^2 \alpha}{\partial N_{1c}^2}(\hat N_{1c},N_{2c})<0
	\]
	Consequently, any extremum of the map $N_{1c} \mapsto {\cal R}_{1,2}(N_{1c},N_{2c})$ is a local minimizer, which implies that this map has at most one local minimizer.
\end{proof}

\medskip

Finally, we give conditions for which the minimization of the threshold ${\cal R}_{1,2}$ presents a trichotomy.

\begin{proposition}
	\label{mainprop2}
	Let parameters $\beta_i$, $\gamma$ be such that ${\cal R}_2>{\cal R}_1$ and assume that $N_1$, $N_2$ satisfy  $N_1{\cal R}_2>N_2{\cal R}_1$.
	Then, provided that $\gamma$ is small enough compared to $\lambda_i$ and $\mu_i$, the function $(N_{1c},N_{2c}) \mapsto {\cal R}_{1,2}(N_{1c},N_{2c})$ admits an unique minimum at $(N_{1c}^\star,N_{2c}^\star)$ with $N_{2c}^\star=N_2$. Moreover, one has the following properties.
	\begin{enumerate}
		\item $N_{1c}^\star=0$ if $\eta_2 > 1-\eta_1$,
		\item $N_{1c}^\star=N_1$ if $\eta_1$ and $\eta_2$ are sufficiently small,
		\item there exists $\eta_1$, $\eta_2$ for which  $N_{1c}^\star \in (0,N_1)$.
	\end{enumerate}
\end{proposition}

\begin{proof}
	We first show that the announced properties are satisfied for the approximate function $\tilde {\cal R}_{1,2}$.
	
	\medskip
	
	From Propositions \ref{prop_dR12} and \ref{propmonotony}, we know that $\tilde {\cal R}_{1,2}$ admits an unique minimum at $(\hat N_{1c},\hat N_{2c})$ with $\hat N_{2c}=N_2$. For $N_{2c}=N_2$, the  condition \eqref{cond+a} simply writes $\eta_2 > 1-\eta_1$ which implies from point ii. of Proposition \ref{prop_dR12} that one has $\hat N_{1c}=0$ when this condition is fulfilled. This shows that point 1 is verified for the function  $\tilde {\cal R}_{1,2}$. 
	
	One obtains the limits
	\[
	\lim_{\eta_1, \eta_2 \to 0} A = \frac{{\cal R}_2}{2}- \frac{{\cal R}_1}{2}> 0, \quad 	\lim_{\eta_1, \eta_2 \to 0} B = \frac{{\cal R}_2}{N_2}- \frac{{\cal R}_1}{N_1}> 0
	\]
	which show that numbers $A$ and $B$ are positive when $\eta_1$, $\eta_2$ are small, and thus  one has $\hat N_{1c}=N_1$ from point iii of Proposition \ref{prop_dR12}. This shows that point 2 is verified for the function  $\tilde {\cal R}_{1,2}$.
	
	Take now any $N_{1c} \in (0,N_1)$. When  $\eta_2 > 1-\eta_1$, one has $\frac{\partial \tilde {\cal R}_{1,2}}{\partial N_{1c}}(N_{1c},N_2)>0$, and for $\eta_1$, $\eta_2$ small, $\frac{\partial \tilde {\cal R}_{1,2}}{\partial N_{1c}}(N_{1c},N_2)<0$ is verified. Then, by continuity of the function $\tilde {\cal R}_{1,2}$ with respect to parameters $\eta_1$, $\eta_2$, one deduce that the existence of values $\hat \eta_1$, $\hat \eta_2$ for which $\frac{\partial \tilde {\cal R}_{1,2}}{\partial N_1}(N_{1c},N_2)=0$. As the function $\tilde {\cal R}_{1,2}$  cannot have more than a local extremum (see Proposition \ref{propmonotony}), we deduce that $N_{1c}$ realizes the minimum of  the function $N_{1c} \mapsto \tilde {\cal R}_{1,2}(N_{1c},N_2)$ when $\eta_1=\hat \eta_1$ and $\eta_2=\hat \eta_2$. This shows that point 3 is verified for the function  $\tilde {\cal R}_{1,2}$.
	
	\medskip
	
	Finally, note that the exact threshold ${\cal R}_{1,2}$ amounts to replace in the expression of $\tilde q_{12}$,  $\tilde q_{21}$ the numbers $\eta_i$ by $\frac{\lambda_i+\gamma}{\lambda_i+\mu_i+\gamma}$ , which is continuous with respect to $\gamma$ and equal to $\eta_i$ for $\gamma=0$. By continuity of $\tilde {\cal R}_{1,2}$ with respect to $\tilde q_{12}$,  $\tilde q_{21}$ , we deduce that uniqueness of the minimizer of ${\cal R}_{1,2}$ and properties 1. to 3. are also fulfilled by  the function $(N_{1c},N_{2c})\mapsto {\cal R}_{1,2}$, provided that $\gamma$ is small enough.
\end{proof}

\section{Numerical illustration}
\label{secillu}

We consider two territories of same population size $N=N_1=N_2$ with different transmission rates such that one has ${\cal R}_1 < 1  < {\cal R}_2$ (values are given in Table \ref{tableparam}). Typically, some precautionary measures (such as social distance) are taken in the first territory so that the disease cannot spread in this territory if it is closed, while the epidemic can spread in the second territory in absence of communication with territory 1. We aim at studying how the epidemic can die out when commuting occur between territories, depending on the proportions of resident in each population, denoted
\[
p_i=:= \frac{N_{ir}}{N}= 1-\frac{N_{ic}}{N}, \quad (i=1,2)
\]
(in other words, how to obtain ${\cal R}_{1,2}<1$ playing with $p_1$, $p_2$). Note that when $N_1=N_2$, the threshold ${\cal R}_{1,2}$ depends on the proportions $p_1$, $p_2$ independently of $N$.

\begin{table}[ht!]
	\centering
	\begin{tabular}{|c|c|c||c|c|}
		\hline
		$\gamma$ & $\beta_1$ & $\beta_2$ & ${\cal R}_1$ & ${\cal R}_2$\\
		\hline\hline
		$0.3$ & $0.27$ & $0.33$ & $0.9$ & $1.1$\\
		\hline
	\end{tabular}
\smallskip
	\caption{Characteristics numbers of the epidemic \label{tableparam}}
\end{table}

Conditions of Proposition \ref{mainprop2} are satisfied provided that commuting parameters $\lambda_i$, $\mu_i$ are large enough. We have considered three sets of these parameters, given in Table \ref{tablecomm}, that correspond to the three possible
situations depicted in Proposition \ref{mainprop2}.

\begin{table}[ht!]
	\centering
	\begin{tabular}{|c|c|c|c|c||c|c|}
		\hline
		case & $\lambda_1$ & $\mu_1$ &  $\lambda_2$ & $\mu_2$ & $\eta_1$ & $\eta_2$\\
		\hline\hline
		A & $10$ & $10$ & $10$ & $1$ & $0.5$ & $0.9090909$\\
		\hline
		B & $10$ & $100$ & $10$ & $100$ & $0.009901$ & $0.009901$\\
		\hline
		C & $10$ & $10$ & $10$ & $70$ & $0.5$ & $0.125$\\
		\hline
	\end{tabular}
\smallskip
	\caption{Three sets of commuting parameters \label{tablecomm}}
\end{table}
The approximate expression $\tilde {\cal R}_{1,2}$ turns out to be a very good approximation of the exact value  ${\cal R}_{1,2}$, even in case A for which $\gamma$ is not so small compared to $\mu_2$ (see Table \ref{tableapprox}).

\begin{table}[ht!]
	\centering
	\begin{tabular}{|c||c|c|c|}
		\hline
		case & A & B & C\\
		\hline\hline
		$\ds \vphantom{\int}\max_{p_1,p_2}\vert\tilde {\cal R}_{1,2}- {\cal R}_{1,2}\vert$ & $1.9\,10^{-3}$ & $1.4\,10^{-4}$  & $6\,10^{-4}$ \\
		\hline
	\end{tabular}
\smallskip
	\caption{Quality of the approximation $\tilde {\cal R}_{1,2}$ \label{tableapprox}}
\end{table}

Figures \ref{fig1}, \ref{fig2}, \ref{fig3} show families of curves $p_1 \mapsto {\cal R}_{1,2}$ for different values of $p_2 \in [0,1]$. 
One can observe that theses curves possess the properties given by Propositions \ref{prop_dR12} and \ref{propmonotony}:
\begin{itemize}
	\item[-] they are either decreasing, increasing or decreasing down to a minimum and then increasing,
	\item[-]  they are ordered and the lower one is obtained for $p_2=0$ (i.e.~$N_{2c}=N_2$).
\end{itemize}
This last feature is intuitive: the more there are commuters from territory 2 (that spend time in territory 1 where the conditions of transmission disease is lower), the less the epidemic spreads. A way to reduce the value of ${\cal R}_{1,2}$ is thus to encourage commuting towards territory 1 (whatever are the commuting rates). However, the role of the resident population in territory 1 is far less intuitive because it does depends on the commuting rates.
\begin{enumerate}
	\item In case A, commuters from territory 2 return more rarely to home than commuters from territory 1 do. The condition of point 1 of Proposition \ref{mainprop2} is fulfilled. Then, the threshold ${\cal R}_{1,2}$ can be made small (and below $1$) when the proportion of resident in territory 1 is high i.e.~when the inhabitants of territory 1 are encouraged not to commute.
	\item In case B, both commuters return rapidly to their home. This means that the numbers of commuters from one territory present in the other one at a given time is low. Then the condition of point 2 of Proposition \ref{mainprop2} is fulfilled. Here,  it is better to encourage inhabitants of territory 1 to commute to the other territory where the disease spreads yet more easily... which is counter-intuitive at first sight. Indeed, commuters do not spend much time in the other territory, and therefore heuristically have less time to encounter and transmit the disease...
	\item In case C, commuters from territory 2 return more rapidly to home than commuters from territory 1 do, on the opposite of case A. Conditions of points 1 and 2 of  Proposition \ref{mainprop2} are not fulfilled here and we are in an intermediate situation for which point 3 of  Proposition \ref{mainprop2} occurs. It is theoretically possible to have  ${\cal R}_{1,2}<1$ on the condition that the proportion of commuters of territory 1 is well balanced.
\end{enumerate}
Finally, this example shows that changing only the return rates $\mu_1$, $\mu_2$ allows to obtain the three possible scenarios, but other changes could also exhibit them.

\begin{figure}[ht!]
	\centering
	\includegraphics[scale=0.5]{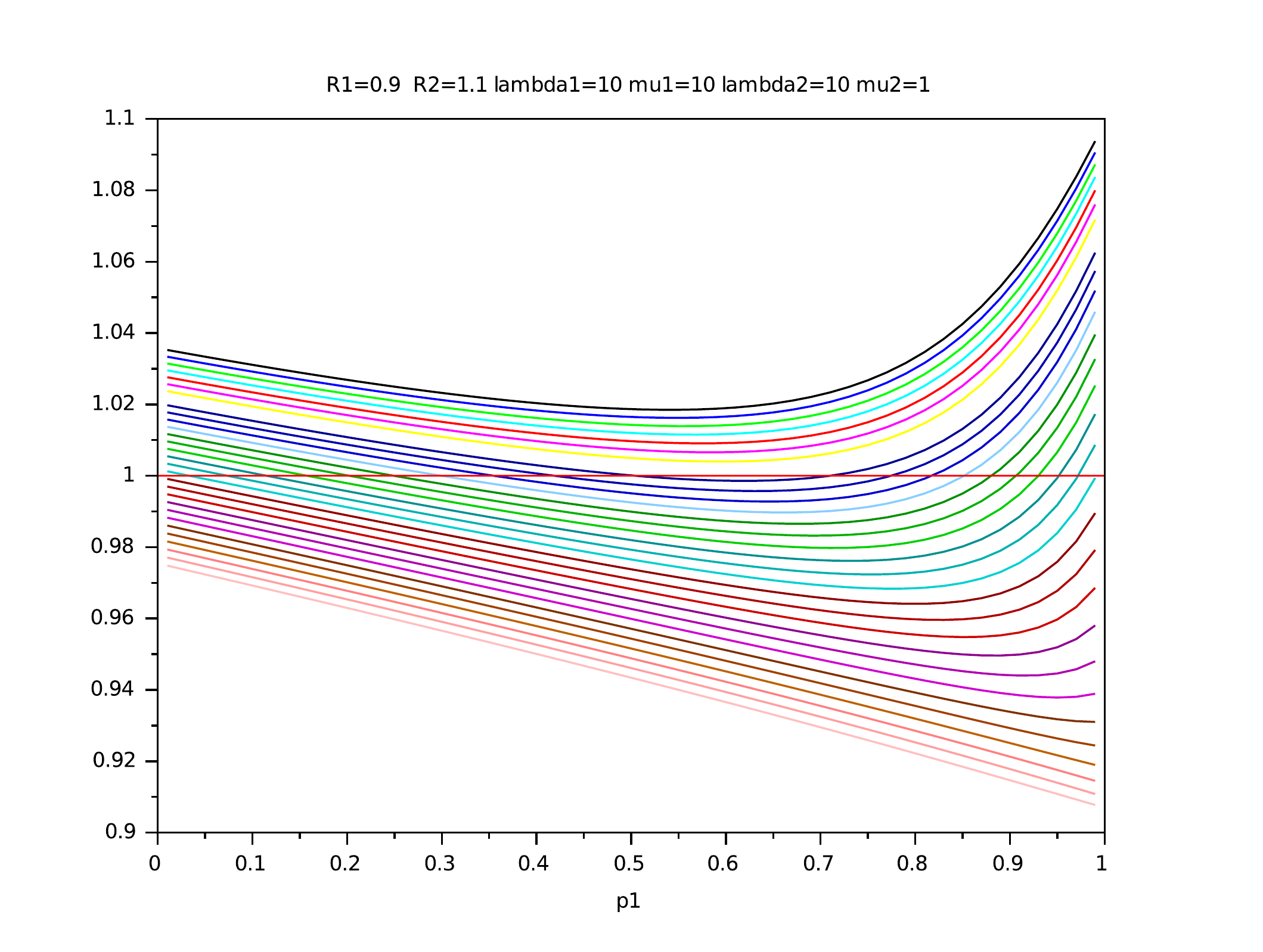}
	\caption{\label{fig1} ${\cal R}_{1,2}$ as a function of $p_1$ in case A (each curve corresponds to a value of $p_2 \in [0,1]$)}
\end{figure}

\begin{figure}[ht!]
	\centering
	\includegraphics[scale=0.5]{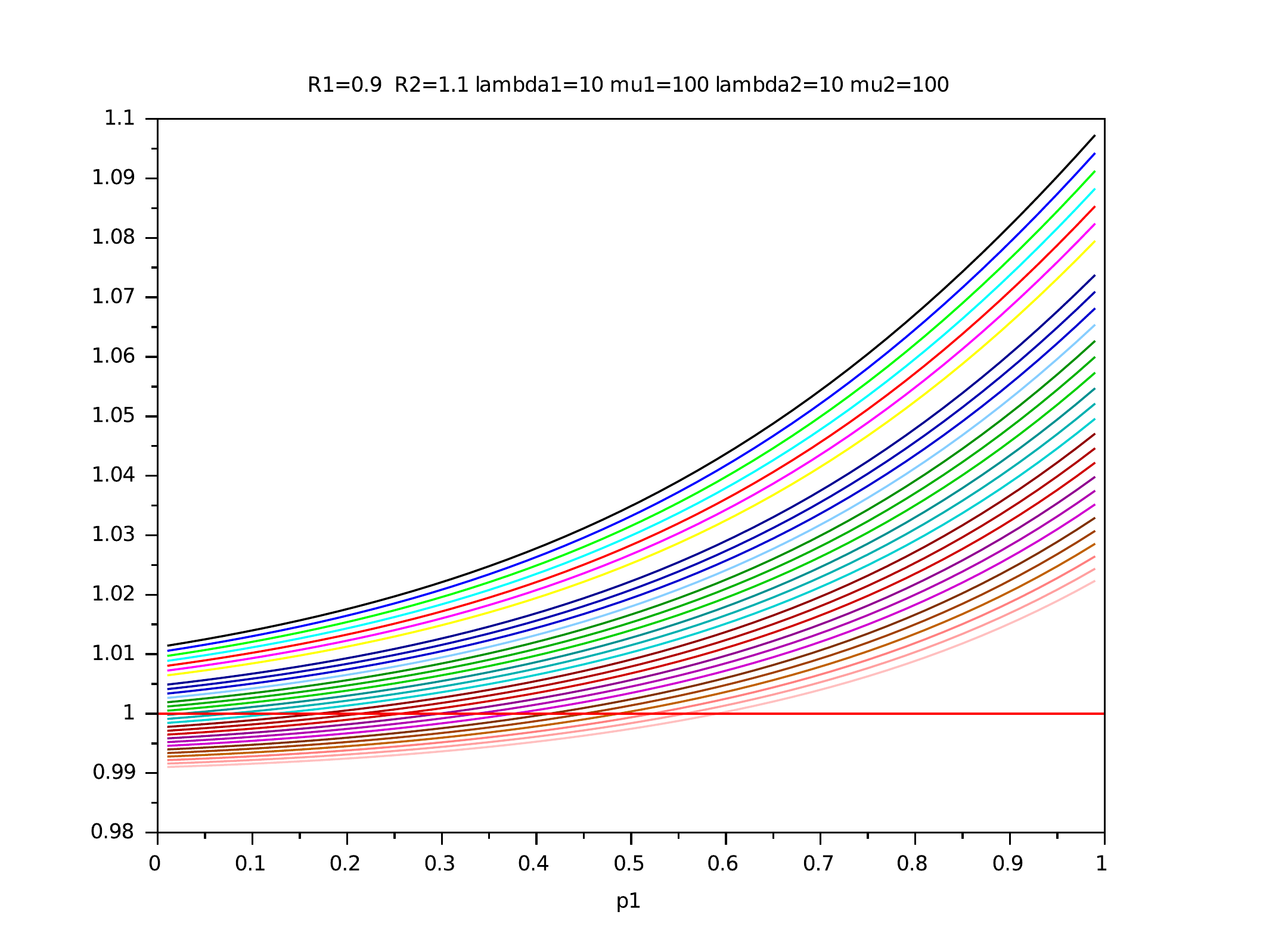}
	\caption{\label{fig2} ${\cal R}_{1,2}$ as a function of $p_1$ in case B (each curve corresponds to a value of $p_2 \in [0,1]$)}
\end{figure}

\begin{figure}[ht!]
	\centering
	\includegraphics[scale=0.5]{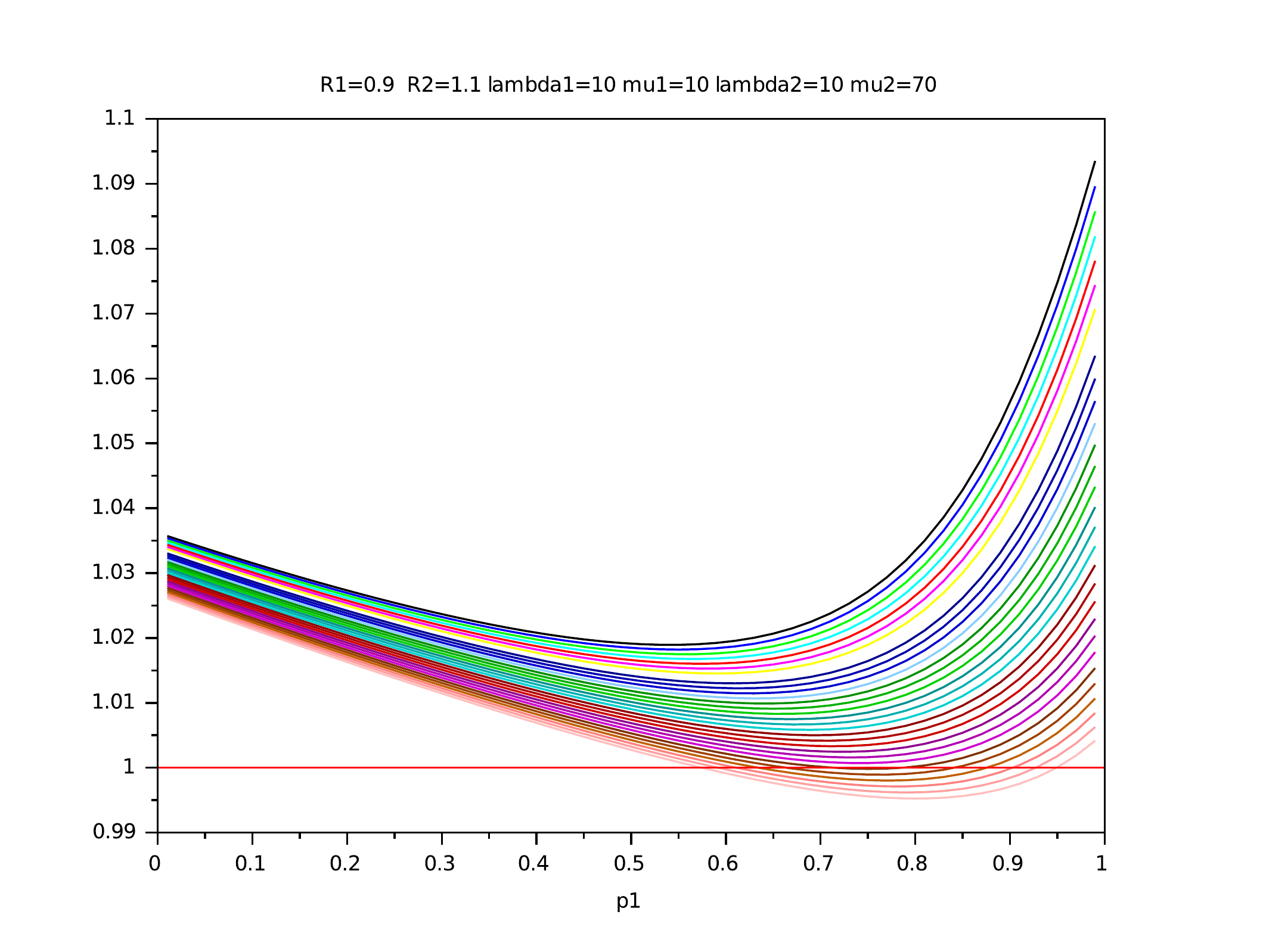}
	\caption{\label{fig3} ${\cal R}_{1,2}$ as a function of $p_1$ in case C (each curve corresponds to a value of $p_2 \in [0,1]$)}
\end{figure}

\newpage

\section{Conclusion}
\label{secconc}

In this work, we have been able to provide an explicit expression of the reproduction number, although the model is in dimension $18$. This expression has allowed us to study its minimization with respect to the proportions of permanently resident populations in each patch. We discovered a trichotomy of cases, with some counter intuitive situations. In each case, it is always beneficial to have commuters traveling to a safer city where the transmission rate is lower. However, for the safer city, three situations occurs: 
\begin{itemize}
	\item[-] either it is better to avoid commuting to the other city,
	\item[-] or on the opposite encouraging commuting to the more risky city reduces the reproduction number,
	\item[-] and in a third case there exists an optimal intermediate proportion of commuters of the safer city which minimizes the epidemic threshold.
\end{itemize}
In some sense, the permanently resident populations, which have been ignored in former modeling, can play an hidden role in an epidemic outbreak. This is illustrated on an example for which only right proportions of commuters (or permanently resident) avoid the outbreak. This suggests that counter-intuitive situations may also occur when considering networks with more than two nodes. 
The present study focuses on the reproduction number and how it can be reduced. The impacts of resident proportions on other epidemiological characteristics, such as the peak level or the finite size, may be the matter a future work.
The extension of the present results to more general networks is also a future perspective.

\section*{Acknowledgments}
The authors are grateful for the support of the French platform MODCOV19, and the Algerian Government for the PhD grant of Ismail Mimouni.

The authors thank the anonymous referee to let us know the alternative approach to obtain the expression of the reproduction number, mentioned in Remark \ref{remark}.



\end{document}